\documentclass[a4paper]{amsart}
\usepackage[centertags]{amsmath}
\usepackage[breaklinks=true,colorlinks]{hyperref}
\usepackage[textwidth=15cm,hcentering]{geometry}
\usepackage{amsfonts}
\usepackage{amssymb}
\usepackage{amsthm}
\usepackage{newlfont}
\usepackage{amscd}
\usepackage{amsmath,amscd}
\usepackage{graphicx}
\usepackage[all]{xy}
\usepackage{verbatim}

\usepackage{color}

\usepackage{eucal}

\makeatletter
\def\@tocline#1#2#3#4#5#6#7{\relax
\ifnum #1>\c@tocdepth 
  \else
    \par \addpenalty\@secpenalty\addvspace{#2}%
\begingroup \hyphenpenalty\@M
    \@ifempty{#4}{%
      \@tempdima\csname r@tocindent\number#1\endcsname\relax
 }{%
   \@tempdima#4\relax
 }%
 \parindent\z@ \leftskip#3\relax \advance\leftskip\@tempdima\relax
 \rightskip\@pnumwidth plus4em \parfillskip-\@pnumwidth
 #5\leavevmode\hskip-\@tempdima #6\nobreak\relax
 \ifnum#1<0\hfill\else\dotfill\fi\hbox to\@pnumwidth{\@tocpagenum{#7}}\par
 \nobreak
 \endgroup
  \fi}
\makeatother

\makeatletter
\ifcsname phantomsection\endcsname
    \newcommand*{\qrr@gobblenexttocentry}[5]{}
\else
    \newcommand*{\qrr@gobblenexttocentry}[4]{}
\fi
\newcommand*{\addsubsection}{%
    \addtocontents{toc}{\protect\qrr@gobblenexttocentry}%
    \subsection}
\makeatother

\usepackage{lipsum}

\newcommand{\barr}{\overline}

\newcommand{\QQ}{\mathbb{Q}}
\newcommand{\ZZ}{\mathbb{Z}}

\newcommand{\PP}{\mathbb{P}}

\newcommand{\NN}{\mathbb{N}}
\newcommand{\cA}{\mathcal{A}}

\newcommand{\cD}{\mathcal{D}}

\newcommand{\cG}{\mathcal{G}}

\newcommand{\cS}{\mathcal{S}}
\newcommand{\cT}{\mathcal{T}}

\newtheorem{thm}{Theorem}[subsection]

\newtheorem{cor}[thm]{Corollary}

\newtheorem{lem}[thm]{Lemma}
\newtheorem{prop}[thm]{Proposition}

\theoremstyle{definition}
\newtheorem{define}[thm]{Definition}

\theoremstyle{remark}
\newtheorem{rem}[thm]{Remark}

\DeclareMathOperator{\Ext}{Ext}

\DeclareMathOperator{\im}{Im}

\DeclareMathOperator{\Top}{Top}

\DeclareMathOperator{\Set}{\cS et}

\DeclareMathOperator{\Grp}{\cG rp}

\DeclareMathOperator{\Hom}{Hom}

\DeclareMathOperator{\precolim}{colim}

\DeclareMathOperator{\des}{des}

\DeclareMathOperator{\Id}{Id}
\DeclareMathOperator{\id}{id}
\DeclareMathOperator{\e}{e}

\DeclareMathOperator{\C}{C}

\DeclareMathOperator{\inc}{inc}
\DeclareMathOperator{\Ab}{Ab}
\DeclareMathOperator{\AbC}{\cA b}

\newcommand{\op}{\mathrm{op}}

\def\colim{\mathop{\precolim}}

\def\lrar{\longrightarrow}

\title{The abelianization of inverse limits of groups}

\author{
Ilan Barnea \and Saharon Shelah}

\address{Department of Mathematics\\
Hebrew University of Jerusalem\\
Jerusalem\\
Israel}
\email{ilanbarnea770@gmail.com}

\address{Department of Mathematics\\
Hebrew University of Jerusalem\\
Jerusalem\\
Israel}
\email{shelah@math.huji.ac.il}

\thanks{The second author was partially supported by European Research Council grant 338821. Publication 1084.}


\begin{document}

\begin{abstract}
The abelianization is a functor from groups to abelian groups, which is left adjoint to the inclusion functor. Being a left adjoint, the abelianization functor commutes with all small colimits. In this paper we investigate the relation between the abelianization of a limit of groups and the limit of their abelianizations. We show that if $\mathcal{T}$ is a countable directed poset and $G:\mathcal{T}\longrightarrow\mathcal{G} rp$ is a diagram of groups that satisfies the Mittag-Leffler condition, then the natural map
$$\mathrm{Ab}(\lim_{t\in\mathcal{T}}G_t)\longrightarrow\lim_{t\in\mathcal{T}}\mathrm{Ab}(G_t)$$
is surjective, and its kernel is a cotorsion group. In the special case of a countable product of groups, we show that the Ulm length of the kernel does not exceed $\aleph_1$.
\end{abstract}

\maketitle
\tableofcontents

\section*{Introduction}

The abelianization functor is a very fundamental and widely used construction in group theory and other mathematical fields. This is a functor
$$\Ab:\Grp\lrar\AbC,$$
from the category of groups to the category of abelian groups, equipped with a natural projection map
$$\pi_G:G\lrar\Ab(G),$$
for every group $G$. This construction is universal in the sense
that for any group $G$, any abelian group $A$ and any morphism of groups $f:G\to A$, there is a unique morphism of (abelian) groups $\barr{f}:\Ab(G)\lrar A$ such that the following diagram commutes:
  $$\xymatrix{G\ar[r]^{\pi_G}\ar[dr]_f & \Ab(G)\ar[d]^{\barr{f}}\\
                                & A.}$$

Expressed in the language of category theory, the above universal property implies that the functor $\Ab:\Grp\lrar \AbC$ is left adjoint to the inclusion functor $\inc: \AbC\lrar \Grp$.
Being a left adjoint, the functor $\Ab$ commutes with all small colimits. That is, given any small category $\cD$, and any functor (diagram) $F:\cD\lrar \Grp$, the natural morphism
$$\colim_{d\in\cD}\Ab(F_d)\lrar\Ab(\colim_{d\in\cD}F_d)$$
is an isomorphism. However, the abelianization functor certainly does not commute with all small limits. That is, given a
small category $\cD$ and a diagram $G:\cD\lrar \Grp$, the natural morphism
$$\Ab(\lim_{d\in\cD}G_d)\lrar\lim_{d\in\cD}\Ab(G_d).$$
need not be an isomorphism. Since this is a morphism of abelian groups, a natural way to ``measure" how far it is from being an isomorphism is to consider its kernel an cokernel. Thus, a natural question is whether the kernel and cokernel of the natural map above can be any abelian groups, or are there limitations?

In this paper we consider the case where the diagram category $\cD$ is a countable directed poset, considered as a category which has a single morphism $t\to s$ whenever $t\geq s$. We show
\begin{thm}[see Theorem~\ref{t:S equation}]\label{t:main cotor}
  Let $\cT$ be a countable directed poset and let $G:\cT\lrar\Grp$ be a diagram of groups that satisfies the Mittag-Leffler condition. Then the natural map
$$\Ab(\lim_{t\in\cT}G_t)\lrar\lim_{t\in\cT}\Ab(G_t)$$
is surjective, and its kernel is a cotorsion group.
\end{thm}

Recall that $G$ satisfies the Mittag-Leffler condition if for every $t\in\cT$ there exists $s\geq t$ such that for every $r\geq s$ we have $$\im (G_s\to G_t)=\im (G_r\to G_t).$$
If $G$ has surjective connecting homomorphisms or $G$ is diagram of finite groups then $G$ satisfies Mittag-Leffler condition.

\begin{rem}
  A generalization of Theorem \ref{t:main cotor} to arbitrary  directed posets, as well as the question of which cotorsion groups can appear as the kernel of a map as in Theorem \ref{t:main cotor}, will be addressed in future papers.
\end{rem}

Recall that a group $G$ is called \emph{perfect} if $\Ab(G)=0$. We thus obtain the following corollary:

\begin{cor}
  Let $\cT$ be a countable directed poset and let $G:\cT\lrar\Grp$ be a diagram of perfect groups that satisfies the Mittag-Leffler condition. Then
$\Ab(\lim_{t\in\cT}G_t)$
is a cotorsion group.
\end{cor}

Cotorsion groups are abelian groups $A$ that satisfy $\Ext(\QQ,A)=0$ (or, equivalently, $\Ext(F,A)=0$ for any torsion free abelian group). That is, an abelian group $A$ is cotorsion iff for every group $B$, containing $A$ as a subgroup, and satisfying $B/A\cong\QQ$, we have that $A$ is a direct summand in $B$. This is a very important and extensively studied class of abelian groups. There is a structure theorem, due to Harrison \cite{Har}, which classifies cotorsion groups in terms of a countable collection of cardinals together with a reduced torsion group. We note that the functor
$${\lim}^1:\Grp^\NN\lrar \Set_*$$
is used in the proof of Theorem \ref{t:main cotor}, and it is known that an abelian group is cotorsion if and only if it is in the image of this functor when restricted to $\AbC^\NN$ (see \cite{WaHu} and Section \ref{ss:lim}).

Our main tool in proving Theorem \ref{t:main cotor} is the following result, which we believe is interesting in its own right:

\begin{thm}[see Theorem~\ref{t:equation gen}]\label{t:equation gen 0}
Let $G:\NN\lrar \Grp$ be a tower of groups. We denote by $G_\omega=\lim_{n\in \NN}G_n$ its limit
and by
$$\phi_n:G_\omega\to G_n$$
the natural map, for every $n\in\NN$.
Suppose that $F\triangleleft G_\omega$ is a normal subgroup such that the quotient $G_\omega/F$ is abelian and for every $n\in\NN$ we have
  $$\phi_n(F)=\phi_n(G_\omega).$$
Then the abelian group $G_\omega/F$ is cotorsion.
\end{thm}

Now suppose that we are given a countable collection of groups $(H_n)_{n\in\NN}$. Then we can construct from them a diagram $G:\NN\lrar \Grp$ by letting $G_n$ be the product $H_1\times\cdots\times H_n,$
for every $n\in\NN$, and $G_{m}\to G_n$ be the natural projection, for every $m\geq n$.
Since for every $n\in\NN$ we have
$$\Ab(\prod_{i\leq n}H_i)\cong\prod_{i\leq n}\Ab(H_i),$$
we see that the natural map in Theorem \ref{t:main cotor} is the natural map
$$\Ab(\prod_{i\in\NN}H_i)\lrar \prod_{i\in\NN}\Ab(H_i).$$

Clearly all the structure maps in the diagram $G$ are surjective, so by Theorem \ref{t:main cotor} we have that the natural map is surjective, and its kernel is a cotorsion group. But in this case we are able to say more about the kernel of the natural map. Namely, we show
\begin{thm}[see Theorem \ref{t:Ulm}]\label{t:main Ulm}
 Let $(H_n)_{n\in\NN}$ be a countable collection of groups. Then the natural map
$$\Ab(\prod_{i\in\NN}H_i)\lrar \prod_{i\in\NN}\Ab(H_i)$$
is surjective, and its kernel is a cotorsion group of Ulm length that does not exceed $\aleph_1$.
\end{thm}

For a definition of the Ulm length of an abelian group see Definition \ref{d:Ulm}.

\begin{rem}
  A generalization of Theorem \ref{t:main Ulm} to arbitrary  products, as well as the question of whether our bound on the Ulm length of the kernel of the map in Theorem \ref{t:main Ulm} is strict or can be improved, will be addressed in future papers.
\end{rem}

\begin{rem}
It is not hard to find examples where the natural map in Theorem \ref{t:main Ulm} (and in particular, in Theorem \ref{t:main cotor}) is not an isomorphism. We explain in remark \ref{r:example} that if $G$ is any group, the natural map
$$\Ab(G^\NN)\lrar \Ab(G)^\NN$$
is an isomorphism if and only if the commutator width of $G$ is finite. This gives a lot of examples where this map is not an isomorphism. For instance, one can take $G$ to be a free group on more then one generator.
\end{rem}

Again, we obtain the immediate corollary:
\begin{cor}
  The abelianization of a countable product of perfect groups is a cotorsion group of Ulm length that does not exceed $\aleph_1$.
\end{cor}

This paper originated from a question posed to the second author by Emmanuel Farjoun, from the field of algebraic topology. For the following discussion, the word space will mean a compactly generated Hausdorff topological space.
It is known, that the homology groups of a homotopy \textbf{colimit} of spaces are computable from the homology groups of the individual spaces, by means of a spectral sequence, while this is not true for the homology groups of a homotopy \textbf{limit}. Farjoun asked what can be said about the natural map from the homology of a homotopy limit to the limit of homologies. Since the \textbf{homotopy} groups of a homotopy limit are computable from the homotopy groups of the individual spaces, by means of a spectral sequence, and we have a natural isomorphism
$$H_1(X)\cong\Ab(\pi_1 X),$$
for every pointed connected space $X$ (see \cite[Corollary 3.6]{GJ}), a good place to start seems to be the investigation of the behaviour of the abelianization functor under limits. And indeed, using our results, we can say something also about Farjoun's question as the following corollary demonstrates:

\begin{cor}\
\begin{enumerate}
  \item Let $X:\NN\lrar\Top_\ast$ be a diagram of pointed connected spaces. Suppose that for every $n\in\NN$ the structure map $X_{n+1}\to X_n$ is a Serre fibration and both $\pi_1 X$ and $\pi_2 X$ satisfy the Mittag-Leffler condition.
   Then the natural map
$$H_1(\lim_{n\in\NN}X_n)\lrar\lim_{n\in\NN}H_1(X_n)$$
is surjective, and its kernel is a cotorsion group.

  \item Let $(Y_n)_{n\in\NN}$ be a countable collection of connected spaces. Then the natural map
$$H_1(\prod_{i\in\NN}Y_i)\lrar \prod_{i\in\NN}H_1(Y_i)$$
is surjective, and its kernel is a cotorsion group of Ulm length that does not exceed $\aleph_1$.
\end{enumerate}
\end{cor}

\begin{proof}
We begin with (1). Since $X$ is a tower of pointed fibrations, we have for every $i\geq 0$ an exact sequence (see, for instance, \cite[VI Proposition 2.15]{GJ})
   $$*\lrar {\lim}_{n\in\NN}^{1}\pi_{i+1}X_n \lrar \pi_{i}\lim_{n\in\NN} X_n \lrar \lim_{n\in\NN}\pi_{i}X_n \lrar *.$$
   Since $\pi_1 X$ and $\pi_2 X$ satisfy the Mittag-Leffler condition, we have that
   $${\lim}_{n\in\NN}^1\pi_{1}X_n\cong {\lim}_{n\in\NN}^1\pi_{2}X_n\cong *.$$
   Thus, from the exact sequence in the case $i=0$ we can deduce that $\lim_{n\in\NN} X_n$ is connected, so we have a natural isomorphism
   $$H_1(\lim_{n\in\NN} X_n)\cong\Ab(\pi_1\lim_{n\in\NN} X_n),$$
   while from the exact sequence in the case $i=1$ we obtain a natural isomorphism
   $$\pi_{1}\lim_{n\in\NN} X_n \cong \lim_{n\in\NN}\pi_{1}X_n.$$
   It is not hard to see that these isomorphisms fit into a commutative diagram
   $$\xymatrix{H_1(\lim_{n\in\NN} X_n)\ar[r]^{\cong}\ar[d] & \Ab(\lim_{n\in\NN} \pi_1 X_n)\ar[d]\\
   \lim_{n\in\NN} H_1(X_n)\ar[r]^{\cong} & \lim_{n\in\NN} \Ab(\pi_1 X_n),}$$
   where the vertical maps are the natural ones. Now the result follows from Theorem \ref{t:main cotor}.

   The proof of (2) is identical to (1), but we use Theorem \ref{t:main Ulm} instead of Theorem \ref{t:main cotor}.
\end{proof}

\subsection{Organization of the paper}
In Section \ref{s:prelim} we recall some necessary background from the theory of groups and abelian groups. In Section \ref{s:seq lim} we prove Theorem \ref{t:main cotor} and in Section \ref{s:prod} we prove Theorem
\ref{t:main Ulm}.

\subsection{Notations and conventions}\label{ss:conv}
We denote by $\NN$ the set of natural numbers (including $0$), by $\ZZ$ the ring of integers and by $\QQ$ the field of rational numbers. We denote by $\PP$ the set of prime natural numbers and for every $p\in\PP$, we denote by $\ZZ_p$ the ring of $p$-adic integers. Whenever we treat a ring as an abelian group we mean its underlying additive group.

If $G$ is a group, we denote its unit element by $\e_G$ or just by $\e$ if the group is understood.

If $\cT$ is a small partially ordered set, we view $\cT$ as a small category which has a single morphism $u\to v$ whenever $u\geq v$.

\subsection{Acknowledgments}
We would like to thank Emmanuel Dror-Farjoun for suggesting the problem and for his useful comments on the paper.
We also thank Manfred Dugas, Brendan Goldsmith, Daniel W. Herden and Lutz Str\"ungmann for helpful correspondences. Finally, we would like to thank the referee for his useful comments and suggestions.

\section{Preliminaries}\label{s:prelim}
In this section we recall some necessary background from the theory of groups and abelian groups. The material is taken mainly from \cite{Fu1}, \cite{GJ} and \cite{Wei}. We assume basic familiarity with category theory \cite{ML}.

\subsection{The abelianization functor}\label{ss:comm}
Let $G$ be a group. Recall that the \emph{commutator} of two elements $x$ and $y$ of $G$ is
$$[x,y]:=x^{-1}y^{-1}xy\in G.$$
The \emph{commutator subgroup} of $G$, denoted $\C(G)$, is the subgroup of $G$ generated by the commutators. It is easily seen that $\C(G)$ is a normal subgroup of $G$ and we have
$$\C(G)=\{[x_1,y_1]\cdots [x_n,y_n]\:\:|\:\:n\in \NN,\:\:x_i,y_i\in G\}.$$

The \emph{commutator length} of an element $x$ in $\C(G)$ is the least integer $n$, such that $x$ is a product of $n$ commutators in $G$. The \emph{commutator width} of $G$ is the supremum of the commutator lengths of the elements in $\C(G)$ (which is either a natural number or $\infty$).

If $f:G\to H$ is a group homomorphism, then
$$f|_{\C(G)}:\C(G)\lrar\C(H),$$
so $\C$ can be naturally extended to a functor from the category of groups to itself
$$\C:\Grp\lrar\Grp.$$

The abelianization of $G$ is defined to be the quotient group
$$\Ab(G):=G/\C(G).$$
Clearly $\Ab(G)$ is an abelian group and if $f:G\to H$ is a group homomorphism, we have an induced map
$$\Ab(f):\Ab(G)\lrar\Ab(H).$$
This turns $\Ab$ into a functor from the category of groups to the category of abelian groups
$$\Ab:\Grp\lrar\AbC.$$
We have a projection map
$$\pi_G:G\lrar\Ab(G),$$
and this map is natural in $G$ in the sense that it defines a natural transformation of functors from groups to groups
$$\pi:\Id_{\Grp}\lrar\Ab.$$


\subsection{The Mittag-Leffler condition and $\lim^1$}\label{ss:lim}

Let $\cT$ be a poset. Recall that according to our convention (see Section \ref{ss:conv}), we view $\cT$ as a small category which has a single morphism $t\to s$ whenever $t\geq s$.
Recall also that $\cT$ is called directed if for every $t,s\in \cT$ there exists $r\in \cT$ such that $r\geq t$ and $r\geq s$.

\begin{define}\label{d:ML}
  Let $\cT$ be a directed poset and let $G:\cT\lrar\Grp$ be a diagram of groups. Then $G$ is said to satisfy the Mittag-Leffler condition if for every $t\in\cT$ there exists $s\geq t$ such that for every $r\geq s$ we have
  $$\im (G_s\to G_t)=\im (G_r\to G_t).$$
\end{define}

Clearly every directed diagram of groups with surjective connecting homomorphisms satisfies Mittag-Leffler condition, as does every directed diagram of finite groups.

Let $G:\cT\lrar\Grp$ be a directed diagram of groups. For every $t\in \cT$ we define
$$G'_t:=\bigcap_{s\geq t}\im (G_s\to G_t)< G_t.$$
Clearly, by restriction of the structure maps, we can lift $G'$ into a diagram $G':\cT\lrar \Grp$. It is not hard to see that
we have a natural isomorphism $\lim G'\cong\lim G$. If $G$ satisfies the Mittag-Leffler condition then all the structure maps of $G'$ are surjective.

Let $G:\NN\lrar \Grp$ be a diagram of shape $\NN$ in the category of groups. Such a diagram will be called a \emph{tower of groups}. For every $n\in\NN$ we have a unique morphism $n+1\to n$ in $\NN$, and we define
$$g_{n+1}:=G(n+1\to n):G_{n+1}\to G_n.$$

Let $\lim G$ denote the following subgroup of the product group:
$$\lim G:=\{(x_n)\in\prod_{n=0}^\infty G_n\:\:|\:\:\forall n\in\NN \:\:.\:\:g_{n+1}(x_{n+1})=x_{n}\}.$$
Restricting the natural projections
$$\prod_{m=0}^\infty G_m\lrar G_n,$$
we obtain maps $\lim G\lrar G_n$, for every $n\in \NN$.
Then $\lim G$, together with the maps $\lim G\lrar G_n$, is a limit of the diagram $G$ in the category of groups $\Grp$.

If for every $n\in\NN$ the group $G_n$ is abelian, that is, if $G:\NN\lrar\AbC$, then clearly $\lim G$ is also abelian. In fact, $\lim G$, together with the maps $\lim G\lrar G_n$, is a limit of $G$ in the category of \textbf{abelian} groups $\AbC$.

\begin{define}
  We define an equivalence relation $\sim$ on the \textbf{pointed set} $\prod_{n=0}^{\infty}G_n$ by letting $(x_n)\sim(y_n)$ iff there exist $(a_n)\in \prod_{n=0}^{\infty}G_n$ such that
  $$(y_n)=(a_n x_n g_{n+1}(a_{n+1})^{-1}).$$
  We now define $\lim^1 G$ to be the quotient pointed set
  $${\lim}^1 G:=\prod_{n=0}^{\infty}G_n/\sim.$$
\end{define}

It is not hard to see that $\lim^1$ can be lifted to a functor from the category of towers of groups to the category of pointed sets.

\begin{rem}\label{r:lim1 Ab}
If for every $n\in\NN$ the group $G_n$ is abelian, that is, if $G:\NN\lrar \AbC$, then $\lim G$ and $\lim^1 G$ are naturally identified with the kernel and cokernel of the homomorphism
$$\prod_{n=0}^{\infty}G_n\xrightarrow{\partial} \prod_{n=0}^{\infty}G_n,$$
defined by
$$(a_n)\mapsto(a_n-g_{n+1}(a_{n+1})).$$
In particular, in this case $\lim^1 G$ is an abelian group (and not just a pointed set). Furthermore, $\lim^1$ is a functor from the category of towers of abelian groups to the category of abelian groups.
\end{rem}

\begin{thm}[{\cite[Lemma VI.2.12]{GJ}}]\label{t:lim1}
  The functor $\lim^1$ has the following properties:
  \begin{enumerate}
    \item If $0\to G_1\to G_2\to G_3\to 0$ is a short exact sequence of towers of groups, then we have an exact sequence
        $$0\to\lim G_1\to\lim G_2\to\lim G_3\to{\lim}^1 G_1\to{\lim}^1 G_2\to{\lim}^1 G_3\to 0,$$
        where the morphisms (except the middle one) are the ones induced by the functors $\lim$ and $\lim^1$.
    \item If $G$ is a tower of groups that satisfies the Mittag-Leffler condition, then $\lim^1 G=0$.
  \end{enumerate}
\end{thm}

\begin{rem}
  Exactness of the sequence appearing in the conclusion of Theorem \ref{t:lim1} (1) just means that the image of every map equals the inverse image of the next map at the special point. If $0\to G_1\to G_2\to G_3\to 0$ is a short exact sequence of towers of \textbf{abelian} groups, then this exact sequence becomes an exact sequence of abelian groups.
\end{rem}

\subsection{Divisibility and the Ulm length}

\begin{define}
Let $A$ be an abelian group.
\begin{enumerate}
  \item If $n\in\ZZ$, we denote $nA:=\{na\:|\:a\in A\}.$
  \item If $a\in A$ and $n\in\ZZ$, we denote by $n|a$ the statement that $a\in nA$.
  \item We say that $A$ is \emph{divisible} if
  $$\bigcap_{n=1}^\infty nA=A,$$
  that is, if for every $a\in A$ and every $n\geq 1$ we have $n|a$.
\end{enumerate}
\end{define}

\begin{define}\label{d:Ulm}
  Let $A$ be an abelian group. We define, recursively, for every ordinal $\lambda$, a subgroup $ A^\lambda\subseteq A$, called the \emph{$\lambda$th Ulm subgroup} of $A$, by:
  \begin{enumerate}
    \item $ A^0:=A$.
    \item For every ordinal $\lambda$ we define $A^{\lambda+1}:=\bigcap_{n=1}^\infty n A^\lambda$.
    \item For every limit ordinal $\delta$ we define $A^{\delta}:=\bigcap_{\lambda<\delta} A^{\lambda}$.
  \end{enumerate}
  Clearly we have defined a function $A^{(-)}$ from ordinals to sets, that is monotone decreasing and continuous. In particular, $A^{(-)}$ stabilizes, that is, there exists an ordinal $\lambda$ such that $A^{\lambda+1}= A^\lambda$. The smallest such ordinal is called the \emph{Ulm length} of $A$, and is denoted by $u(A)$. We always have $u(A)\leq |A|$.

Clearly, for every ordinal $\lambda$, we have that $ A^\lambda$ is divisible iff $u(A)\leq\lambda$. We also have that $A^{u(A)}$ is the biggest divisible subgroup of $A$, and we denote $D_A:=A^{u(A)}$.
\end{define}

\begin{thm}[{\cite[Theorem 24.5]{Fu1}}]\label{t:inj}
  Let $A$ be an abelian group. Then the following conditions are equivalent:
  \begin{enumerate}
    \item $A$ is divisible.
    \item $A$ is an injective $\ZZ$-module.
    \item $A$ is a direct summand of every group containing $A$.
  \end{enumerate}
\end{thm}

\begin{define}
  Let $A$ be an abelian group. Then $A$ is called \emph{reduced} if $A$ has no divisible subgroups other then $0$.
\end{define}

An easy consequence of Definition \ref{d:Ulm} and Theorem \ref{t:inj} is:

\begin{thm}\label{t:div red}
Let $A$ be an abelian group. Then there exists a reduced subgroup $R_A$ of $A$, unique up to isomorphism, such that
$$A=D_A\oplus R_A.$$
\end{thm}

The discussion on divisibility and the Ulm length can also be done for every prime separately.

\begin{define}
  Let $A$ be an abelian group and $p$ a prime number. We say that $A$ is \emph{$p$-divisible} if $pA=A,$ that is, if for every $a\in A$ we have $p|a$.
\end{define}

\begin{define}
  Let $A$ be an abelian group and $p$ a prime number. We define, recursively, for every ordinal $\lambda$, a subgroup $p^\lambda A\subseteq A$, by:
  \begin{enumerate}
    \item $p^0 A:=A$.
    \item For every ordinal $\lambda$ we define $p^{\lambda+1}A:=p(p^\lambda A)$.
    \item For every limit ordinal $\delta$ we define $p^{\delta}A:=\bigcap_{\lambda<\delta}p^{\lambda} A$.
  \end{enumerate}
  Clearly we have defined a function $p^{(-)}A$ from ordinals to sets, that is monotone decreasing and continuous. In particular, $p^{(-)}A$ stabilizes, that is, there exists an ordinal $\lambda$ such that $p^{\lambda+1}A=p^\lambda A$. The smallest such ordinal is called the \emph{$p$-length} of A, and is denoted $\l_p(A)$.

Clearly, for every ordinal $\lambda$, we have that $p^\lambda A$ is $p$-divisible iff $l_p(A)\leq\lambda$. We also have that $p^{\l_p(A)} A$ is the biggest $p$-divisible subgroup of $A$.
\end{define}

\subsection{The $\Ext$ functor and cotorsion groups}

\begin{define}
  Let $A$ and $C$ be abelian groups. We denote by $\Ext(C,A)$ the set of equivalence classes of short exact sequences of the form
  $$0\to A\to B\to C\to 0,$$
  where an exact sequence as above is called equivalent to an exact sequence
  $$0\to A\to B'\to C\to 0,$$
  if there exists an isomorphism $B\to B'$ such that the following diagram commutes:
  $$\xymatrix{0\ar[r] & A\ar[r]\ar@{=}[d] & B\ar[r]\ar[d] & C\ar[r]\ar@{=}[d] & 0\\
   0\ar[r] & A\ar[r] & B\ar[r] & C\ar[r] & 0.}$$

   The set $\Ext(C,A)$ can be given the structure of an abelian group by defining
   $$[0\to A\xrightarrow{f} B\xrightarrow{g} C\to 0]+[0\to A\xrightarrow{f'} B'\xrightarrow{g'} C\to 0],$$
   to be
   $$[0\to A\xrightarrow{f\prod f'} B\oplus B'\xrightarrow{g\coprod g'} C\to 0].$$
   The zero element in $\Ext(C,A)$ is the splitting short exact sequence
   $$[0\to A\to A\oplus C\to C\to 0].$$

   It is also possible to lift the above construction to a functor
   $$\Ext:\AbC^{\op}\times\AbC\to \AbC,$$
   using pullbacks and pushouts in the category of abelian groups.
\end{define}

\begin{thm}[{\cite[Application 3.5.10]{Wei}}]\label{t:Ext colim}
    Let $B$ be an abelian group and
    $A:\NN^{\op}\to\AbC$ a diagram of injections of abelian groups:
     $$A_0\hookrightarrow A_1\hookrightarrow\cdots.$$
     Then there is a short exact sequence
     $$0\lrar {\lim}^1_n\Hom(A_n,B)\lrar \Ext(\colim_n A_n,B) \lrar \lim_n\Ext(A_n,B)\lrar 0.$$
  \end{thm}

\begin{define}\label{d:cotor}
  An abelian group $G$ is called \emph{cotorsion} if  $\Ext(\QQ,G)=0$.
\end{define}

In other words, an abelian group $G$ is cotorsion iff for every abelian group $H$, that contains $G$ as a subgroup, such that $H/G\cong\QQ$, we have that $G$ is a direct summand of $H$.

\section{Countable directed limits}\label{s:seq lim}
The purpose of this section is to prove Theorem \ref{t:main cotor}. We begin with a few preliminary definitions and propositions.

In \cite[Section 22]{Fu1}, Fuchs defines the notion of a \emph{system of equations} over an abelian group. There is an equation for every $i\in I$, with unknowns $(x_j)_{j\in J}$, where $I$ and $J$ can be arbitrary sets. We will only be using this notion with $I=J=\NN$.

\begin{define}
  Let $H$ be an abelian group. A \emph{system of equations} over $H$ is the following data:
  \begin{enumerate}
    \item An $\NN\times\NN$ matrix with entries in $\ZZ$, denoted $(l_{n,m})$, such that for every $n\in\NN$ the set
        $$\{m\in \NN\:\:|\:\:l_{n,m}\neq 0\}$$
        is finite.
    \item An element (vector) $(a_n)$ in $H^\NN$.
  \end{enumerate}

  Note that for every $n\in \NN$ we have an equation
  $$\sum_{m=0}^{\infty}l_{n,m}x_m=a_n,$$
  with unknowns $(x_m)$.

  A \emph{solution}, in $H$, to the system of equations is a vector $(g_n)$ in $H^\NN$ such that substituting $(x_m)=(g_m)$ in the system above yields correct statements.
\end{define}

\begin{prop}\label{p:eqation is cotor}
  Let $H$ be an abelian group. Then the following conditions are equivalent:
\begin{enumerate}
  \item The group $H$ is cotorsion.
  \item Every system of equations over $H$, with matrix
  $$\xymatrix{ 1 & -1 & 0   & 0   & \cdots\\
                 &  1  & -2 & 0   & \cdots\\
                 &     &   1 & -3 & \cdots\\
                 &     &     &  1  & \cdots\\
                 &     &     &     & \cdots}$$
  has a solution in $H$.
  \item Every system of equations over $H$, with a matrix that is upper triangular and has identities in its diagonal, has a solution in $H$.
\end{enumerate}
\end{prop}

\begin{proof}
$(2)\Rightarrow(1)$. It is not hard to see that
$$\QQ:=\colim(\ZZ\xrightarrow{1}\ZZ\xrightarrow{2}\ZZ\xrightarrow{3} \cdots),$$
where $n$ denotes multiplication by $n$. So, by Theorem \ref{t:Ext colim}, there is a short exact sequence
$$0\lrar {\lim}^1_n\Hom(\ZZ,H)\lrar \Ext(\colim_n \ZZ,H) \lrar \lim_n\Ext(\ZZ,H)\lrar 0.$$
There is a natural isomorphism $\Hom(\ZZ,-)\cong \id_{\AbC}$ and, since $\ZZ$ is a free abelian group, we know that $\Ext(\ZZ,H)\cong 0$. Thus we obtain an isomorphism
$${\lim}^1 (\cdots\xrightarrow{3}H\xrightarrow{2}H\xrightarrow{1}H) \cong\Ext(\QQ,H).$$
By implementing condition (2) above, we see that for every $(a_n)\in H^\NN$ there exists $(b_n)\in H^\NN$ such that for every $n\in\NN$ we have
$b_n=a_n+(n+1)b_{n+1}$. By Remark \ref{r:lim1 Ab}, it follows that
$${\lim}^1 (\cdots\xrightarrow{3}H\xrightarrow{2}H\xrightarrow{1}H) \cong 0.$$
Thus, $\Ext(\QQ,H)\cong 0$ so $H$ is cotorsion (see Definition \ref{d:cotor}).

$(1)\Rightarrow(3)$. Let $(l_{n,m})$, $(b_n)$ be a system of equations over $H$, and suppose that $l_{n,n}=1$ and $l_{n,m}=0$ if $n>m$. Let us recall from \cite{Fu1} the notion of an \emph{algebraically compact group}. By \cite[Theorem 38.1]{Fu1}, an algebraically compact group can be defined as an abelian group $G$ such that every system of equations over $G$, for which every finite subsystem has a solution in $G$, also has a global solution in $G$. Since $H$ is cotorsion, we know, by \cite[Proposition 54.1]{Fu1}, that there exists an algebraically compact group $G$ and a surjective homomorphism $p:G\to H$. For every $n\in\NN$ let us choose $c_n\in G$, such that $p(c_n)=b_n$. Then $(l_{n,m})$, $(c_n)$ is a system of equations over $G$, and it is easy to see that every finite subsystem of it has a solution in $G$. (If $N\in\NN$ we can find a solution to the first $N$ equations as follows: First define $x_{N+1},x_{N+2},\cdots$ to be any elements in $G$. Now define $x_N\in G$ according to the $N$'th equation. Then define $x_{N-1}\in G$ according to the $(N-1)$'th equation, and so on.) Thus, we have a solution $(g_n)$ to this system of equations in $G$. Now it is easily seen that $(p(g_n))$ is a solution to our original system of equations in $H$.

$(3)\Rightarrow(2)$ is obvious so we are done.
\end{proof}

Let $H:\NN\lrar \Grp$ be a tower of groups. If $m\leq n$, then we have a unique morphism $n\to m$ in $\NN$. We define
$$\phi_{m,n}:=H(n\to m):H_n\to H_m.$$
We denote by $H_\omega$ the limit
$$H_\omega=\lim_{n\in \NN}H_n$$
and by
$$\phi_n:H_\omega\to H_n$$
the natural map, for every $n\in\NN$ (see Section \ref{ss:lim}).

\begin{lem}\label{l:reisha e}
  Suppose that $F< H_\omega$ is a subgroup and for every $n\in\NN$ we have
  $$\phi_n(F)=\phi_n(H_\omega).$$
  Then for every $f\in H_\omega$ and every $n\in\NN$, there exists $\bar{f}\in H_\omega$, such that:
  \begin{enumerate}
    \item $F\bar{f}=Ff$.
    \item For every $i<n$ we have $\phi_i(\bar{f})=e_{H_i}$.
  \end{enumerate}
\end{lem}

\begin{proof}
Let $f\in H_\omega$ and let $n\in\NN$. If $n=0$ let us choose $\bar{f}:=f$. Suppose $n>0$.
By the hypothesis of the lemma we know that
  $$\phi_{n-1}(F)=\phi_{n-1}(H_\omega).$$
  Since $\phi_{n-1}(f)\in \phi_{n-1}(H_\omega)$, we obtain that there exists $f'\in F$ such that
  $$\phi_{n-1}(f')=\phi_{n-1}(f).$$
  We now define
  $$\bar{f}:=(f')^{-1}f\in H_\omega,$$
  so clearly $F\bar{f}=Ff$.

  Now let $i<n$. The following diagram commutes
  $$\xymatrix{H_\omega\ar[r]^{\phi_{n-1}} \ar[dr]_{\phi_{i}} & H_{n-1}\ar[d]^{\phi_{i,n-1}}\\
    &  H_i .}$$
    It follows that
    $$\phi_i(f')=\phi_{i,n-1}(\phi_{n-1}(f'))=
    \phi_{i,n-1}(\phi_{n-1}(f))=\phi_i(f),$$
    so we have
    $$\phi_i(\bar{f})=(\phi_i (f'))^{-1}\phi_i(f)=e_{H_i},$$
    which finishes the proof of our lemma.
\end{proof}

\begin{thm}\label{t:equation gen}
  Suppose that $F\triangleleft H_\omega$ is a normal subgroup such that $T:=H_\omega/F$ is abelian and for every $n\in\NN$ we have
  $$\phi_n(F)=\phi_n(H_\omega).$$
Then the abelian group $T$ is cotorsion.
\end{thm}

\begin{proof}
Let $(f_n)$ be an element in $T^\NN$. By Proposition \ref{p:eqation is cotor}, we need to show that there exists an element $(g_n)$ in $T^\NN$ such that for every $n\in \NN$ we have
$$g_n=f_n+(n+1) g_{n+1}.$$

Let $n\in\NN$. We have
$$f_n\in T=H_\omega/F,$$
so we can choose $f'_n\in H_\omega$ such that
$$[f'_n]:=F f'_n=f_n.$$
By Lemma \ref{l:reisha e}, there exists $\bar{f}_n\in H_\omega$, such that:
  \begin{enumerate}
    \item $[\bar{f}_n]=[f'_n]=f_n$.
    \item For every $l<n$ we have $\phi_l(\bar{f}_n)=e_{H_l}$.
  \end{enumerate}

For every $n>l$ we define $\bar{g}_{n,l}:=e_{H_l}$.

Now, let $l\geq 0$ be fixed. We have defined $\bar{g}_{n,l}\in H_l$ for every $n>l$. Let us now define $\bar{g}_{n,l}\in H_l$ for every $n\leq l$ recursively, using the formula
$$\bar{g}_{n,l}= \phi_l(\bar{f}_n)\bar{g}_{n+1,l}^{n+1}.$$

That is, we define
$$\bar{g}_{l,l}= \phi_l(\bar{f}_l)\bar{g}_{l+1,l}^{l+1}=
\phi_l(\bar{f}_l),$$
$$\bar{g}_{l-1,l}= \phi_l(\bar{f}_{l-1})\bar{g}_{l,l}^{l}=
\phi_l(\bar{f}_{l-1})\phi_l(\bar{f}_l)^{l},$$
and so on.

We have now defined $\bar{g}_{n,l}\in H_l$ for every $n,l\in \NN$, and clearly the formula
$$\bar{g}_{n,l}= \phi_l(\bar{f}_n)\bar{g}_{n+1,l}^{n+1}$$
is now satisfied for every $n,l\in \NN$. (Note that whenever $n>l$ we have $\phi_l(\bar{f}_n)=e_{H_l}$ by (2) above.)
For every $n\in\NN$ we now define
  $$\bar{g}_{n}:=(\bar{g}_{n,l})_{l\in\NN}\in\prod_{l\in \NN}H_l.$$

Recall from Section \ref{ss:lim}, that
$$H_\omega\cong\{(x_l)_{l\in \NN}\in\prod_{l\in \NN}H_l\:\:|\:\:\forall l\in\NN \:\:.\:\:\phi_{l,l+1}(x_{l+1})=x_{l}\}.$$
We want to show that for every $n\in\NN$ we actually have $\bar{g}_{n}\in H_\omega$, that is, that for every $n,l\in\NN$ we have
$$\phi_{l,l+1}(\bar{g}_{n,l+1})=\bar{g}_{n,l}.$$
Clearly, this follows from the following lemma, taking $i=l+2$:
\begin{lem}
  Let $l\in\NN$ be fixed. Then for every $i\leq l+2$ and every $n>l-i+1$ we have
  $$\phi_{l,l+1}(\bar{g}_{n,l+1})=\bar{g}_{n,l}.$$
\end{lem}

\begin{proof}
We prove the lemma by induction on $i$. When $i=0$ then $n>l+1$ so we have
$$\bar{g}_{n,l+1}=e_{H_{l+1}},\:\:\bar{g}_{n,l}=e_{H_{l}}$$
and the lemma is clear. Now suppose we have proven the lemma for some $i<l+2$, and let us prove it for $i+1$. Let $n>l-i$.
We know that
$$\bar{g}_{n,{l+1}}= \phi_{l+1}(\bar{f}_n)
\bar{g}_{n+1,{l+1}}^{n+1}.$$
It follows that
$$\phi_{l,l+1}(\bar{g}_{n,{l+1}})=\phi_{l,l+1}(\phi_{l+1}(\bar{f}_n))
\phi_{l,l+1}(\bar{g}_{n+1,{l+1}})^{n+1}.$$
Since $n+1>l-i+1$, we can use the induction hypothesis to obtain
$$\phi_{l,l+1}(\bar{g}_{n,{l+1}})=\phi_{l}(\bar{f}_n)
\bar{g}_{n+1,{l}}^{n+1}=\bar{g}_{n,l},$$
which proves our lemma.
\end{proof}

Let $n\in\NN$. We have shown that $\bar{g}_{n}\in H_\omega$. We define
$$g_n:=[\bar{g}_{n}]\in H_\omega/F=T.$$
For every $l\in \NN$ we have an equality in $H_l$
$$\bar{g}_{n,l}= \phi_l(\bar{f}_n)\bar{g}_{n+1,l}^{n+1}.$$
Thus in $H_\omega$ we have
$$\bar{g}_{n}= \bar{f}_n\bar{g}_{n+1}^{n+1}.$$
Passing to equivalence classes we obtain the following equality in $T=H_\omega/F$:
$$[\bar{g}_{n}]= [\bar{f}_n][\bar{g}_{n+1}]^{n+1}.$$
But $T$ is abelian, so in additive notation we obtain
$$g_n=f_n+(n+1)g_{n+1}.$$
which finishes the proof of our theorem.
\end{proof}

We now turn to the main result of this section.

\begin{thm}\label{t:S equation}
  Let $\cT$ be a countable directed poset and let $G:\cT\lrar\Grp$ be a diagram of groups that satisfies the Mittag-Leffler condition. Then the natural map
  $$\rho:\Ab(\lim_{t\in\cT}G_t)\lrar\lim_{t\in\cT}\Ab(G_t)$$
is surjective and its kernel is cotorsion.
\end{thm}

\begin{proof}
Since $\cT$ is a countable directed poset, there exists a cofinal functor $\NN\lrar\cT$. Thus we can assume that $\cT=\NN$.

We have the following short exact sequence of towers of groups:
$$0\to \C(G)\to G\to \Ab(G)\to 0.$$
Thus, by Theorem \ref{t:lim1}, we have an exact sequence
$$0\to\lim \C(G)\to\lim G\to\lim \Ab(G)\to{\lim}^1 C(G)\to{\lim}^1 G\to{\lim}^1 \Ab(G)\to 0.$$
Since $G$ satisfies the Mittag-Leffler condition, it is not hard to see that $\C(G)$ also satisfies the Mittag-Leffler condition. (Note that for every structure map $\phi:G_n\to G_{m}$ we have $\phi(\C(G_n))=\C(\phi(G_n)).)$
Thus, by Theorem \ref{t:lim1}, we have $\lim^1 \C(G)=0$ so we obtain a short exact sequence of groups
$$0\to\lim \C(G)\to\lim G\to\lim \Ab(G)\to 0.$$
In particular, $\lim G\to\lim \Ab(G)$ is surjective, so the map it induces
$\rho:\Ab(\lim G)\to\lim \Ab(G)$ is also surjective.

Moreover, we have a natural inclusion $\lim \C(G)\subseteq\lim G$
and a natural isomorphism
$$\lim G/\lim \C(G)\cong\lim \Ab(G).$$
Under this isomorphism, $\rho$ is becomes the obvious map
$$\lim G/\C(\lim G)\lrar\lim G/\lim \C(G),$$
so we have a natural isomorphism
$$\ker(\rho)\cong \lim \C(G)/\C(\lim G).$$

We know that $\C(\lim G)$ is a normal subgroup of $\lim G$ and thus it is also a normal subgroup of $\lim \C(G)$.

We define $G':\NN\lrar \Grp$ and $\C(G)':\NN\lrar \Grp$ as in the beginning of Section \ref{ss:lim}. Since for every structure map $\phi:G_n\to G_{m}$ we have $\phi(\C(G_n))=\C(\phi(G_n))$, it is not hard to see that
$$\C(G)'=\C(G').$$

Let $n\in\NN$. Since all the structure maps $G'_{m+1}\to G'_m$ are surjective, it follows that the map
  $\phi_{n}:\lim G\cong\lim G'\to G'_{n}$
  is surjective. Thus
  $$\phi_n|_{\C(\lim G)}:\C(\lim G)\to \C(G'_n)$$
  is also surjective and
  $$\im(\phi_n|_{\C(\lim G)})=\C(G'_n).$$
  Since all the structure maps $C(G'_{m+1})\to C(G'_m)$ are surjective, it follows that the map
  $$\phi_n|_{\lim \C(G)}:\lim \C(G)\cong\lim \C(G')\to \C(G'_{n})$$
  is surjective and
   $$\im(\phi_n|_{\lim \C(G)})=\C(G'_n).$$

   Thus, for every $n\in\NN$ we have
  $$\im(\phi_n|_{\C(\lim G)})=\im(\phi_n|_{\lim \C(G)}).$$
   Using Theorem \ref{t:equation gen} with $H:=\C \circ G:\NN\lrar\Grp$, we see that
  $\ker(\rho)\cong\lim \C(G)/\C(\lim G)$
   is cotorsion.
\end{proof}

\section{Countable products}\label{s:prod}
Let $(H_n)_{n\in\NN}$ be a countable collection of groups. We can construct from this collection a diagram $G:\NN\lrar \Grp$ by letting $G_n$ be the product $H_1\times\cdots\times H_n,$
for every $n\in\NN$, and $G_{m}\to G_n$ be the natural projection, for every $m\geq n$.
Since for every $n\in\NN$ we have
$$\Ab(\prod_{i\leq n}H_i)\cong\prod_{i\leq n}\Ab(H_i),$$
we see that the natural map in Theorem \ref{t:S equation} becomes
$$\rho:\Ab(\prod_{i\in\NN}H_i)\lrar \prod_{i\in\NN}\Ab(H_i).$$
Clearly all the structure maps are surjective so, by Theorem \ref{t:S equation}, $\rho$ is surjective and $\ker(\rho)$ is cotorsion. The purpose of this section is to show that $\ker(\rho)$ cannot be any cotorsion group in this case. Namely, by Harrison's structure theorem for cotorsion groups \cite{Har}, we know that any torsion group is the torsion part of some cotorsion group. Since a torsion group can have arbitrary large Ulm length, it follows that the same is true for a cotorsion group. However, we show in Theorem \ref{t:Ulm} that $u(\ker(\rho))\leq \aleph_1$.

We begin with a few preliminary definitions and propositions.

\begin{prop}\label{p:p-divisible}
  Let $A$ be an abelian group and $p$ a prime number. Suppose that $(y_m)_{m\in \NN}$ is an element in $A^\NN$ such that for every $m\in\NN$ we have $y_m=py_{m+1}$. Then $y_0\in p^{\l_p(A)} A$.
\end{prop}

\begin{proof}
   Clearly it is enough to show that for every ordinal $\lambda$ we have $y_0\in p^\lambda A$. We show this by induction on $\lambda$.

  Clearly $y_0\in p^0 A=A$.
  Let $\lambda$ be an ordinal and suppose we have shown that $y_0\in p^\beta A$, for every $\beta\leq\lambda$. Applying what we have shown to $(y_m)_{m\geq 1}$ we see that $y_1\in p^\lambda A$. Thus we obtain
  $$y_0=py_1\in p(p^\lambda A)=p^{\lambda+1}A.$$
  Now suppose that $\lambda$ is a limit ordinal and we have shown that $y_0\in p^\beta A$ for every $\beta<\lambda$.
  Then
  $$y_0\in \bigcap _{\beta<\lambda}p^\beta A=p^\lambda A,$$
  which finishes the proof by induction.
\end{proof}

\begin{define}
Let $\lambda$ be an ordinal. We define
$$\des(\lambda):=\{(\mu_1,...,\mu_n)\:|\:\text{$n\geq 0$ and $\lambda>\mu_1>\cdots>\mu_n$}\}.$$
Note that $\des(\lambda)$ contains also the empty string $\phi$, corresponding to $n=0$.

Let $\mu=(\mu_1,...,\mu_n)\in\des(\lambda)$. We define
$$\l(\mu):=n\geq 0,$$
$$\min(\mu):=\begin{cases}
\lambda & \text{ if $\mu=\phi$,}\\
\mu_n & \text{ if $\mu\neq\phi$,}
                    \end{cases}$$
and if $n\geq m\geq 0$ we define $\mu|_{m}:=(\mu_1,...,\mu_{m})$.
\end{define}

The following proposition gives some motivation for the previous definition:

\begin{prop}\label{p:des}
  Let $H$ be an abelian group and let $\lambda$ be an ordinal. Then for every $x\in p^\lambda H$ there exist
  $$\bar{x}=(x_\mu)_{\mu\in\des(\lambda)}
  \in\prod_{\mu\in\des(\lambda)} p^{\min(\mu)}H,$$
  such that the following hold:
  \begin{enumerate}
    \item $x_\phi=x.$
    \item If $\mu\in\des(\lambda)$ and $\l(\mu)=n>0$, then $px_\mu=x_{\mu|_{n-1}}.$
  \end{enumerate}
\end{prop}

\begin{proof}
  We define $x_\mu\in p^{\min(\mu)}H$ for every $\mu\in\des(\lambda)$ recursively, relative to $\l(\mu)$.

  Suppose first that $\l(\mu)=0$. Then $\mu=\phi$ and we define  $$x_\phi:=x\in p^{\min(\phi)}H=p^{\lambda}H.$$

  Let $n\geq 0$ and suppose that we have defined $x_\mu\in p^{\min(\mu)}H$ for every $\mu\in\des(\lambda)$ with $\l(\mu)\leq n$, in such a way that condition (2) above holds where it is defined.

  Now let $\mu\in\des(\lambda)$ such that $\l(\mu)=n+1$. Clearly $\min(\mu)<\min(\mu|{n})$ so $\min(\mu)+1\leq\min(\mu|{n})$ and we have
  $$x_{\mu|{n}}\in p^{\min(\mu|{n})} H\subseteq p^{\min(\mu)+1} H=p(p^{\min(\mu)} H).$$
  Thus there exist $x_{\mu}\in p^{\min(\mu)}H$ such that $px_{\mu}=x_{\mu|{n}}$.
\end{proof}

\begin{prop}\label{p:des logic}
Let $N$ be an infinite set and $\kappa$ a cardinal such that $\kappa>|N|$. Suppose we are given $k_\mu\in N$, for every $\mu\in\des(\kappa)$.
Then there exists a sequence of triples $(k_n,S_n,\mu(n))_{n\in \NN}$, such that for every $n\in \NN$:
\begin{enumerate}
  \item $k_n\in N$.
  \item $S_n\subseteq\kappa$ and $|S_n|=\kappa$.
  \item $\mu(n)=(\mu(n)_\alpha)_{\alpha\in S_n}$ and for every $\alpha\in S_n$ we have
      \begin{enumerate}
        \item $\mu(n)_\alpha\in\des(\kappa)$.
        \item $l(\mu(n)_\alpha)=n+1$.
        \item $\min(\mu(n)_\alpha)=\alpha$.
        \item For every $l\leq n$ we have $k_{\mu(n)_\alpha|_{l+1}}=k_l$.
      \end{enumerate}
\end{enumerate}
\end{prop}

\begin{proof}
  We define the triple $(k_n,S_n,\mu(n))$ recursively with $n$.

  We begin with $n=0$. For every $\alpha<\kappa$ we have $(\alpha)\in \des(\kappa)$, so $k_{(\alpha)}\in  N$. Since $\kappa$ is a cardinal and $\kappa>|N|$, there exists $k_0\in N$ such that
  $$|\{\alpha<\kappa\:|\:k_{(\alpha)}=k_0\}|=\kappa.$$
  We define
  $$S_0:=\{\alpha<\kappa\:|\:k_{(\alpha)}=k_0\},$$
  and for every $\alpha\in S_0$ we define
  $$\mu(0)_\alpha:=(\alpha)\in\des(\kappa).$$
  Clearly, for every $\alpha\in S_0$ we have
  \begin{enumerate}
        \item $l(\mu(0)_\alpha)=1$.
        \item $\min(\mu(0)_\alpha)=\alpha$.
        \item $k_{\mu(0)_\alpha|_{1}}=k_{(\alpha)}=k_0$.
      \end{enumerate}

    Now let $m\in \NN$ and suppose we have defined a triple $(k_n,S_n,\mu(n))$ for every  $n\leq m$ such that the conditions in the proposition are satisfied where they are defined. Let us define the triple  $(k_{m+1},S_{m+1},\mu(m+1))$.

     First we define recursively a strictly increasing function $f_m:\kappa\to S_m$ such that for every $\alpha<\kappa$ we have $\alpha<f_m(\alpha)$.

   We define
    $$f_m(0):=\min(S_m\setminus\{\min(S_m)\}).$$
    Clearly $0<f_m(0)$.

    Let $\beta<\kappa$ and suppose we have defined $f_m(\alpha)\in S_m$ for every $\alpha\leq \beta$. Let us now define $f_m(\beta+1)\in S_m$. We have
    $f_m(\beta)<\kappa$. Since $\kappa$ is a cardinal, it follows that $|f_m(\beta)|<\kappa$ and thus
    $$|S_m\setminus f_m(\beta)|=|\{\lambda\in S_m\:|\:\lambda\geq f_m(\beta)\}|=\kappa.$$
    In particular
    $$\{\lambda\in S_m\:|\:\lambda>f_m(\beta)\}\neq \phi$$
    and we can define
    $$f_m(\beta+1):=\min\{\lambda\in S_m\:|\:\lambda>f_m(\beta)\}\in S_m.$$
    Clearly we have
    $$f_m(\beta+1)>f_m(\beta)\geq \beta+1.$$

    Suppose that $\beta<\kappa$ is a limit ordinal and we have defined $f_m(\alpha)\in S_m$ for every $\alpha<\beta$.
    Since $\kappa$ is a cardinal, we have that $|f_m(\alpha)|<\kappa$ for every $\alpha<\beta$, and also $|\beta|<\kappa$. Thus
    $$|\bigcup_{\alpha<\beta}f_m(\alpha)|<\kappa,$$
    so
    $$|S_m\setminus \bigcup_{\alpha<\beta}f_m(\alpha)|=|\{\lambda\in S_m\:|\:\forall\alpha<\beta.\lambda\geq f_m(\alpha)\}|=\kappa.$$
    Since $|\beta|<\kappa$ we obtain in particular that
    $$\{\lambda\in S_m\:|\:\forall \alpha<\beta.\lambda>f_m(\alpha)\}\setminus\{\beta\}\neq \phi$$
    and we can define
    $$f_m(\beta):=\min(\{\lambda\in S_m\:|\:\forall \alpha<\beta.\lambda>f_m(\alpha)\}\setminus\{\beta\})\in S_m.$$
    For every $\alpha<\beta$ we thus have
    $$f_m(\beta)>f_m(\alpha)>\alpha,$$
    so $f_m(\beta)\geq\beta$. But $f_m(\beta)\neq\beta$ so we obtain
    $f_m(\beta)>\beta$.
    This finishes our recursive definition of $f_m:\kappa\to S_m$.

    For every $\alpha<\kappa$ we have $(\mu(m)_{f_m(\alpha)},\alpha)\in \des(\kappa)$, so $k_{(\mu(m)_{f_m(\alpha)},\alpha)}\in  N$. Since $\kappa$ is a cardinal and $\kappa>|N|$, there exists $k_{m+1}\in N$ such that
    $$|\{\alpha<\kappa\:|\:k_{(\mu(m)_{f_m(\alpha)},\alpha)}=k_{m+1}\}| =\kappa.$$
  We define
  $$S_{m+1}:=
  \{\alpha<\kappa\:|\:k_{(\mu(m)_{f_m(\alpha)},\alpha)}=k_{m+1}\},$$
  and for every $\alpha\in S_{m+1}$ we define
  $$\mu(m+1)_\alpha:=(\mu(m)_{f_m(\alpha)},\alpha)\in\des(\kappa).$$
  Let $\alpha\in S_{m+1}$. Using the induction hypothesis, we have:
  \begin{enumerate}
        \item $l(\mu(m+1)_\alpha)=l(\mu(m)_{f_m(\alpha)})+1=m+2$.
        \item $\min(\mu(m+1)_\alpha)=\alpha$.
      \end{enumerate}
 Let $l\leq m+1$. If $l=m+1$ we have
 $$k_{\mu(m+1)_\alpha|_{l+1}}=k_{\mu(m+1)_\alpha}=k_{m+1}=k_l,$$
 while if $l\leq m$ we have, using the induction hypothesis,
 $$k_{\mu(m+1)_\alpha|_{l+1}}=k_{\mu(m)_{f'_m(\alpha)}|_{l+1}}=k_l,$$
  which finishes the proof of our proposition.
\end{proof}

\begin{thm}\label{t:lenght}
  The natural map
  $$\rho:\Ab(\prod_{i\in\NN}H_i)\lrar \prod_{i\in\NN}\Ab(H_i)$$
is surjective and $\ker(\rho)$ is cotorsion and satisfies $l_p(\ker(\rho))\leq \aleph_1$ or every prime $p$.
\end{thm}

\begin{proof}
By Theorem \ref{t:S equation}, $\rho$ is surjective and $\ker(\rho)$ is cotorsion.

Now let $p\in\PP$. For convenience of notation let us denote
$S:=\ker(\rho).$
Recall from the proof of Theorem \ref{t:S equation} that
  $$S\cong\lim_{n\in \NN}\C(G_n)/\C(\lim_{n\in \NN}G_n) .$$
  Since for every $n\in \NN$ we have
  $$\C(G_n)\cong \C(H_1\times\cdots\times H_n)\cong \C(H_1)\times\cdots\times \C(H_n),$$
  we see that
  $$S\cong\prod_{n\in \NN}\C(H_n)/\C(\prod_{n\in \NN}H_n) .$$
We need to show that $p^{\aleph_1}S$ is $p$-divisible. It is clearly enough to show $p^{\aleph_1}S\subseteq p^{\l_p(S)} S$.

So let $x\in p^{\aleph_1}S$. We define
$$\bar{x}=(x_\mu)_{\mu\in\des(\aleph_1)}
  \in\prod_{\mu\in\des(\aleph_1)} p^{\min(\mu)}S,$$
  as in Proposition \ref{p:des}.

For every $\mu\in\des(\aleph_1)$ we have
$$x_\mu\in p^{\min(\mu)}S\subseteq S=\prod_{n\in \NN}\C(H_n)/\C(\prod_{n\in \NN}H_n),$$
so let us choose a representative
$$f_\mu=(f_\mu(n))_{n\in\NN}\in\prod_{n\in\NN}\C(H_n)$$
such that $[f_\mu]=x_\mu.$

Let $\mu\in\des(\aleph_1)$ with $l(\mu)=n>0$. Then by Proposition \ref{p:des} we have (in multiplicative notation)
$$x^p_\mu=x_{\mu|_{n-1}}.$$
Thus
$$x_{\mu|_{n-1}}^{-1}x_\mu^p=\e\in S,$$
so
$$f_{\mu|_{n-1}}^{-1}f_\mu^p\in\C(\prod_{n\in \NN}H_n).$$
Let $k_\mu\in\NN$ be the commutator length of $f_{\mu|_{n-1}}^{-1}f_\mu^p$, so there exist
$$g_{\mu,t}=(g_{\mu,t}(n))_{n\in\NN}\in\prod_{n\in\NN}H_n$$
for every $t<2k_\mu$ such that
$$f_{\mu|_{n-1}}^{-1}f_\mu^p=\prod_{l<k_\mu}[g_{\mu,2l},g_{\mu,2l+1}].$$

Let us define $k_\phi:=0$. By Proposition \ref{p:des logic}, applied for the set $N=\NN$, the cardinal $\kappa=\aleph_1$, and $(k_\mu)_{\mu\in\des(\aleph_1)}$ defined above, we see that there exists a sequence of triples $(k_n,S_n,\mu(n))_{n\in \NN}$, such that for every $n\in \NN$ we have:
\begin{enumerate}
  \item $k_n\in \NN$.
  \item $S_n\subseteq\aleph_1$ and $|S_n|=\aleph_1$.
  \item $\mu(n)=(\mu(n)_\alpha)_{\alpha\in S_n}$ and for every $\alpha\in S_n$ we have
      \begin{enumerate}
        \item $\mu(n)_\alpha\in\des(\aleph_1)$.
        \item $l(\mu(n)_\alpha)=n+1$.
        \item $\min(\mu(n)_\alpha)=\alpha$.
        \item For every $l\leq n$ we have $k_{\mu(n)_\alpha|_{l+1}}=k_l$.
      \end{enumerate}
\end{enumerate}

For every $m\in\NN$ we define $\alpha_m:=\min(S_m)$, and we define
$$h_m=(h_m(n))_{n\in\NN}\in\prod_{n\in\NN}\C(H_n)$$
by
$$h_m(n):=\begin{cases}
\e_{H_n} & \text{ if $n<m$,}\\
f_{\mu(n)_{\alpha_n}|_m}(n) & \text{ if $n\geq m$.}
                    \end{cases}$$
For every $m>0$ and $t<2k_m$ we define
$$d_{m,t}=(d_{m,t}(n))_{n\in\NN}\in\prod_{n\in\NN}H_n$$
by
$$d_{m,t}(n):=\begin{cases}
\e_{H_n} & \text{ if $n<m$,}\\
g_{\mu(n)_{\alpha_n}|_m,t}(n) & \text{ if $n\geq m$.}
                    \end{cases}$$

Let $n>m\geq 0$. Since $\mu(n)_{\alpha_n}|_{m+1}\in\des(\aleph_1)\setminus\{\phi\}$, we have as above
$$f_{\mu(n)_{\alpha_n}|_{m}}^{-1}f_{\mu(n)_{\alpha_n}|_{m+1}}^p=
\prod_{l<k_{\mu(n)_{\alpha_n}|_{m+1}}}[g_{\mu(n)_{\alpha_n}|_{m+1},2l}, g_{\mu(n)_{\alpha_n}|_{m+1},2l+1}].$$
In particular, we get an equality in $H_n$:
$$f_{\mu(n)_{\alpha_n}|_{m}}(n)^{-1}f_{\mu(n)_{\alpha_n}|_{m+1}}(n)^p=
\prod_{l<k_{\mu(n)_{\alpha_n}|_{m+1}}}[g_{\mu(n)_{\alpha_n}|_{m+1},2l}(n), g_{\mu(n)_{\alpha_n}|_{m+1},2l+1}(n)].$$
But $n\geq m+1$ so we obtain
$$h_m(n)^{-1}h_{m+1}(n)^p=\prod_{l<k_m}[d_{m+1,2l}(n),d_{m+1,2l+1}(n)].$$
For every fixed $m\in\NN$, the equality above holds for almost all $n\in\NN$. Passing to equivalence classes in
$$S=\prod_{n\in \NN}\C(H_n)/\C(\prod_{n\in \NN}H_n)$$
we thus obtain
$$[h_m^{-1}h_{m+1}^p]=
[\prod_{l<k_m}[d_{m+1,2l},d_{m+1,2l+1}]]=\e,$$
or
$$[h_m]=[h_{m+1}]^p.$$
By Proposition \ref{p:p-divisible} we obtain
$$[h_0]\in p^{l_p(S)}S.$$
But for every $n\in\NN$ we have $h_0(n)=f_\phi(n)$ so $h_0=f_\phi$ and
$$x=x_\phi=[f_\phi]=[h_0]\in p^{l_p(S)}S,$$
as required.
\end{proof}

We now wish to prove a variant of Theorem \ref{t:lenght} which uses the Ulm length instead of the $p$-length.

\begin{prop}\label{p:u l p}
   Let $p$ be a prime number and let $G$ be an abelian group that is a module over the $p$-adic integers. Then for every ordinal $\lambda$ we have $u(G)\leq{\lambda}$ iff $l_p(G)\leq\omega\lambda$.
\end{prop}

\begin{proof}
It is shown in \cite[on page 154]{Fu1} that  $$G^\lambda=\bigcap_{p\in \PP}p^{\omega\lambda}G.$$
Since $G$ be a module over the $p$-adic integers we know that $G$ is $q$-divisible for every prime $q\neq p$, so we otain
$$G^\lambda=p^{\omega\lambda}G.$$
Using transfinite induction it is easily seen that $G^{\lambda}$ is a sub $\ZZ_p$-module of $G$, so $G^\lambda$ is $q$-divisible for every prime $q\neq p$.
Thus, $G^\lambda$ is divisible, iff
it is $p$-divisible. We now see that the following statements are equivalent:
  \begin{enumerate}
    \item $u(G)\leq{\lambda}$.
    \item $p^{\omega\lambda} G=G^\lambda$ is divisible.
    \item $p^{\omega\lambda} G=G^\lambda$ is $p$-divisible.
    \item $l_p(G)\leq\omega\lambda$.
  \end{enumerate}
\end{proof}

\begin{prop}\label{p:u max l p}
  Let $G$ be a cotorsion group and let $\lambda$ be an ordinal. Then $u(G)\leq{\lambda}$ iff $l_p(G)\leq\omega\lambda$ for every $p\in\PP.$
\end{prop}

\begin{proof}
By decomposing $G$ into its divisible part and reduced part, as in Theorem \ref{t:div red}, we may assume that $G$ is reduced.

Since $G$ is reduced and cotorsion, $G$ can be written as a product of the form
  $G\cong\prod_{p\in\PP}G_p,$
  where every $G_p$ is a  module over the $p$-adic integers (see \cite[on page 234]{Fu1}). Thus, using Proposition \ref{p:u l p}, we see that the following statements are equivalent:
  \begin{enumerate}
    \item $u(G)\leq{\lambda}$.
    \item $G^\lambda$ is divisible.
    \item $G_p^\lambda$ is divisible for every prime $p$.
    \item $u(G_p)\leq{\lambda}$ for every prime $p$.
    \item $l_p(G)\leq\omega\lambda$ for every prime $p$.
  \end{enumerate}
\end{proof}

We now turn to the main result of this section.

\begin{thm}\label{t:Ulm}
 The natural map
  $$\rho:\Ab(\prod_{i\in\NN}H_i)\lrar \prod_{i\in\NN}\Ab(H_i)$$
is surjective and $\ker(\rho)$ is cotorsion and satisfies $u(\ker(\rho))\leq \aleph_1$.
\end{thm}

\begin{proof}
 By Theorem \ref{t:lenght} $\rho$ is surjective and $\ker(\rho)$ is cotorsion. By Proposition \ref{p:u max l p}, applied for the ordinal $\lambda:=\aleph_1$, we see that $u(\ker(\rho))\leq{\aleph_1}$ iff $l_p(\ker(\rho))\leq\omega\aleph_1=\aleph_1$ for every $p\in\PP.$ Thus the result follows from Theorem \ref{t:lenght}.
\end{proof}

\begin{rem}\label{r:example}
The natural map in Theorem \ref{t:Ulm} (and in particular, in Theorem \ref{t:S equation}) need not be an isomorphism in general. As we explain in the beginning of the proof of Theorem \ref{t:lenght}, the kernel of this natural map is isomorphic to the quotient
$$\prod \C(H_n)/\C(\prod H_n),$$
of the product of the commutators by the commutator of the product. It is not hard to see that the group $\C(\prod H_n)$ consists of those sequences $(h_n)\in \prod \C(H_n)$ for which the commutator length of $h_n$ is bounded, as $n$ ranges through the natural numers. It follows that if $G$ is any group, the natural map
$$\Ab(G^\NN)\lrar \Ab(G)^\NN$$
is an isomorphism if and only if the commutator width of $G$ is finite. This gives a lot of examples where this map is not an isomorphism. For instance, one can take $G$ to be a free group on more then one generator.
\end{rem}


\end{document}